\documentclass[12pt,centertags]{amsart}
\usepackage{amscd}

\usepackage{easybmat}

\usepackage{mathrsfs}
\usepackage{amsfonts}
\usepackage{color}
\usepackage{pifont}
\usepackage{upgreek}
\usepackage{bm}
\usepackage{hyperref}
\usepackage{shorttoc}

\usepackage{amsmath,amstext,amsthm,a4,amssymb,amscd}
\usepackage[mathscr]{eucal}
\usepackage{mathrsfs}
\usepackage{epsf}
\numberwithin{equation}{section}

\def\mF{\mathcal{F}}

\newtheorem{thm}{Theorem}[section]
\newtheorem{lemma}[thm]{Lemma}
\newtheorem{prop}[thm]{Proposition}
\newtheorem{cor}[thm]{Corollary}
\theoremstyle{definition}
\newtheorem{rem}[thm]{Remark}

\theoremstyle{definition}

\newcommand{\be}{\begin{eqnarray}}
\newcommand{\ee}{\end{eqnarray}}

\newcommand{\comment}[1]{}

\usepackage{fancyhdr}
\begin{document}
\title{Flat vector bundles and open coverings}
\author[Huitao Feng]{Huitao Feng$^1$}
\author[Weiping Zhang]{Weiping Zhang$^2$}

\address{Huitao Feng: Chern Institute of Mathematics \& LPMC,
Nankai University, Tianjin 300071,  P. R. China}

\email{fht@nankai.edu.cn}

\address{Weiping Zhang: Chern Institute of Mathematics \& LPMC,
Nankai University, Tianjin 300071, P. R.  China}

\email{weiping@nankai.edu.cn}

\thanks{$^1$~Partially supported by NSFC (Grant No. 11221091, 11271062, 11571184)}

\thanks{$^2$~Partially supported by NSFC (Grant No. 11221091)}

\date{}  
\maketitle
\setcounter{section}{-1}
\begin{abstract}
We establish a generic counting formula for the Euler number of a flat vector bundle  of rank $2n$ over a $2n$ dimensional  closed manifold, in terms of   vertices of transversal open coverings of the underlying manifold.  We use the Mathai-Quillen formalism to prove our result.

\comment{

\begin{flushleft}
Keywords: affine manifold,\quad  Chern conjecture,\quad Mathai-Quillen formalism,\quad flat connection,\quad Euler characteristic
\end{flushleft}

}

\end{abstract}

\tableofcontents

\section{Introduction}\label{S0}

By the celebrated Gauss-Bonnet-Chern theorem \cite{C44} and its ramifications, if a real   vector  bundle $F$  of rank $2n$ over a $2n$ dimensional  closed smooth manifold $M$ admits a flat connection preserving certain Euclidean metric on $F$, then the Euler number  of $F$ vanishes: $\langle e(F),[M]\rangle=0$, where $e(F)$ is the Euler class of $F$.   Milnor \cite{M58} constructs examples of   rank two flat vector bundles over a closed surface of genus $g\geq 2$ such that the corresponding Euler number is nonzero. A closely related question is the famous Chern conjecture on the vanishing of the Euler characteristic of a closed affine manifold (cf.  \cite{K15}).

Our original motivation is to understand the Chern conjecture by using the Mathai-Quillen formalism on the geometric construction of the Thom class \cite{MQ86}. While our efforts \cite{FZ16} on proving the Chern conjecture in this way is not successful, we find that our method leads us to a kind of mystic  counting formula for the Euler number of a flat vector bundle in terms of   transversal open coverings of the underlying manifold. The purpose of this paper is to present this mystic formula.

This paper is organized as follows.  In Section \ref{S1}, we present the exterior algebra version of  the Mathai-Quillen formalism, which allows us to avoid the usual difficulty in computing the Euler class that the   connection in question need  not preserve any  metric on a given vector bundle. In Section \ref{S2},    
we  prove our main result, which is stated as Theorem \ref{t0.2} in Section \ref{s2.1}.

\section{An exterior  algebra version of  the Mathai-Quillen formalism}\label{S1}

In this section, we present  an exterior  algebra version of the Mathai-Quillen formalism \cite{MQ86}. That is,   we replace the spinor bundle considered in \cite{MQ86} by the exterior algebra bundle. This gives a formula for the Euler class  for an arbitrary connection on a Euclidean vector bundle.  

Let  $\pi:E\rightarrow M$ be a real oriented vector bundle of even rank over a   closed oriented manifold $M$. 
Let $\nabla^{E}$ be a connection on $E$. Then it induces  a connection  $\nabla^{\Lambda^*(E^*)}$ on $\Lambda^*(E^*)$, which preserves the ${\bf Z}_2$-splitting $\Lambda^*(E^*)=\Lambda^{\rm even}(E^*)\oplus\Lambda^{\rm odd}(E^*). $

Let $g^{E}$ be a Euclidean metric on $E$. For any $Z\in E$, let $Z^*\in E^*$ be its metric dual. As usual (cf. \cite[Section 4.3]{Z01}), let $c(Z)$  be the Clifford action on $\Lambda^*(E^*)$ defined by 
\begin{align}\label{1.1}
 c(Z)=Z^*\wedge-i_Z , 
\end{align}
where $Z^*\wedge$ (resp. $i_Z$) is the exterior  (resp. interior) multiplication of $Z^*$ (resp. $Z$).
Then
\begin{align}\label{1.2a}
c(Z)^2=-|Z|^2_{g^{E}}.
\end{align}

Let $A$ be  the superconnection \cite{Q85} on $\pi^* \Lambda^*(E^*)$ defined by
\begin{align}\label{1.2}
A=\pi^* \nabla^{\Lambda^*(E^*)} +c(Z),
\end{align}
where $c(Z)$ now acts on $(\pi^*\Lambda^*(E^*))|_Z$ such that for any $\alpha\in \Lambda^*(E^*)$, then 
\begin{align}\label{1.2c}
 c(Z)\left(\left(\pi^*\alpha\right)|_Z\right)=\left.\left(\pi^*(c(Z)\alpha)\right)\right|_Z.
\end{align}

By (\ref{1.2a}) and (\ref{1.2}), $\exp(A^2)$ is of exponential decay along vertical directions of $E$.

\begin{thm}\label{t1.1}
The  closed form $ \left(\frac{1}{2\pi}\right)^{{\rm rk}(E)}\int_{E/M} {\rm tr}_s[\exp(A^2)]$ on $M$ is a representative of $\frac{e(E)}{\widehat A(E)^2}$, where  $e(E)$ is the Euler class of $E$ and
 $\widehat A(E)$ is the Hirzebruch $ \widehat A$-class of $E$.
\end{thm}
\begin{proof}
By   the Chern-Weil theory for superconnections (cf. \cite[Proposition 1.43]{BGV92}) and the above mentioned exponential decay property of $\exp(A^2)$, we see that the cohomology class represented  by $\int_{E/M} {\rm tr}_s[\exp(A^2)]$ does not depend on the choice of $\nabla^{E}$ and $g^E$.  Thus, we may well assume  that  $\nabla^{E}$ preserves $g^{E}$. Then one can follow the strategy in \cite{MQ86}.

In fact, since the computation is local, one may well assume that $E$ is spin. Then one has   the following decomposition  (cf. \cite{LM89}) in terms of  the (Hermitian) spinor bundle $S(E)=S_+(E)\oplus S_-(E)$ associated to $(E,g^E)$,
\begin{align}\label{1.4}
\Lambda^*(E^*)=\left(S_+(E)\oplus S_-(E)\right)\widehat\otimes \left(S^*_+(E)\oplus S^*_-(E)\right),
\end{align}
and $c(Z)$ now acts on $(\pi^*S(E))|_Z$.
Moreover, $\nabla^{\Lambda^*(E^*)}$ decomposes to $\nabla^{S(E)}\otimes {\rm Id}_{S^*(E)}+{\rm Id}_{S(E)}\otimes \nabla^{S^*(E)}$, where $\nabla^{S(E)}$, $\nabla^{S^*(E)}$ are the induced  Hermitian connections on $S(E)$, $S^*(E)$ respectively.  

From  (\ref{1.2}) and  (\ref{1.4}), one gets
\begin{align}\label{1.5}
{\rm tr}_s\left[\exp\left(A^2\right)\right]={\rm tr}_s\left[ \exp\left(\left(\pi^*\nabla^{S(E)}+c(Z)\right)^2\right)\right]\cdot  {\rm tr}_s\left[\exp\left(\left(\pi^*\nabla^{S^*(E)}\right)^2\right)\right]    .
\end{align}

By \cite[Theorem 4.5]{MQ86}, one has
\begin{align}\label{1.6}
\left(\frac{\sqrt{-1}}{2\pi}\right)^{\frac{{\rm rk}(E)}{2}}{\rm tr}_s\left[ \exp\left(\left(\pi^*\nabla^{S(E)}+c(Z)\right)^2\right)\right] = 
(-1)^{\frac{{\rm rk}(E)}{2}}\det\left(\frac{\sinh\left(\pi^*R^{E}/2\right)}{\pi^*R^{E}/2}\right)^{\frac{1}{2}}U,
\end{align}
where $R^{E}=(\nabla^{E})^2$   and $U$ is the Thom form constructed in \cite[(4.7)]{MQ86}. 

By  \cite[Proposition  III.11.24]{LM89}, one has that  
\begin{align}\label{1.7}
 {\rm ch}\left(S^*_+(E)-S^*_-(E)\right) = \frac{e(E)}{ \widehat A(E)}.
\end{align}

By (\ref{1.5})-(\ref{1.7})  and  \cite[Theorem 4.10]{MQ86},  which integrates $U$ along vertical fibers, one sees that
$ 
 \left(\frac{1}{2\pi}\right)^{{\rm rk}(E)}\int_{E/M} {\rm tr}_s[\exp(A^2)]
$ 
  is a representative of $ \frac{e(E)}{ \widehat A(E)^2}$. 
\end{proof}

\begin{cor}\label{t1.2}
If ${\rm rk}(E)=\dim M$, then the following identity holds,
\begin{align}\label{1.7z}
 \langle e(E),[M]\rangle =  \left(\frac{1}{2\pi}\right)^{{\rm rk}(E)}\int_{E} {\rm tr}_s\left[\exp\left(A^2\right)\right].
\end{align}
\end{cor}

\section{A counting formula for the Euler number of flat vector bundles}\label{S2}

In this section, we formulate and prove a generic counting formula for the Euler number of a flat vector bundle in terms of transversal open coverings.  

This section is organized as follows. In Section \ref{s2.1}, we present the basic setting and state our main result as Theorem \ref{t0.2}. In Section \ref{s2.2}, we present an application of   Corollary \ref{t1.2} to the
special  case of flat vector bundles.  In Sections \ref{s2.3}-\ref{s2.5}, we   prove   Theorem \ref{t0.2}.

\subsection{Flat vector bundles and the counting formula}\label{s2.1}

 Let $\pi:F\rightarrow M$ be a real  oriented flat vector bundle of rank $2n$ over a $2n$ dimensional closed oriented  manifold $M$. Let $\nabla^F$ denote the underlying flat connection on $F$. Then there is a finite collection of open coordinate charts $\{(U_\alpha,(x_\alpha^i))\}$ ($ \alpha=1,\,\cdots,\,N$),\footnote{The ordering of $U_\alpha$'s, while arbitrary, will play an important role in what follows.} covering $M$, such that $\nabla^F$ induces over each $U_\alpha$ a canonical  identification
\begin{align}\label{2.1}
F|_{U_\alpha}\simeq U_\alpha\times {\bf R}^{2n}_\alpha. 
\end{align}
For any $U_\alpha$, we fix a coordinate system $(y_\alpha^i)$ of ${\bf R}^{2n}_\alpha$. Then over $U_\alpha\cap U_\beta$, $(y_\alpha^i)$ and $(y_\beta^i)$ transform to each other constant  linearly.  Moreover, the horizontal exterior differential and the vertical exterior differential 
\begin{align}\label{2.2}
d^H=\sum_i dx_\alpha^i\frac{\partial}{\partial x_\alpha^i} ,\ \ \  d^V=\sum_i dy_\alpha^i\frac{\partial}{\partial y_\alpha^i}
\end{align}
on $F|_{U_\alpha}$ 
do not depend on $\alpha$, and we have the decomposition of the exterior differential on the total space  $\mF$ of $F$ such that
\begin{align}\label{2.3}
d=d^H+ d^V
\end{align}
and that
\begin{align}\label{2.3c}
\left(d^H\right)^2=\left( d^V\right)^2=d^Hd^V+d^Vd^H=0.
\end{align}

Without loss of generality, we assume that each  $U_\alpha$ has a smooth boundary $\partial U_\alpha$ and that the $\partial U_\alpha$'s intersect to each other complete transversally. We call such an open covering of $M$ a transversal open covering.\footnote{It is easy to see that such a transversal open covering always exists.}   Then  any point  on $M$ lies on at most $2n$ different  boundaries.  Moreover, the set 
\begin{align}\label{2.8}
{\bf B}=\{p\in M\,:\, p\ {\rm lies\ on}\ 2n\ {\rm different\ boundaries}\}
\end{align}
consists of finite points.

Let $h:[0,1]\rightarrow [0,1]$ be the smooth function   $h(t)=\exp(-1/t^2)$. 

 Let $g^{TM}$ be any metric on $TM$. 
For each $U_\alpha$,
let $r_\alpha$ be the normal geodesic coordinate near  $\partial U_\alpha$.
Let $\rho_\alpha\in C^\infty(M) $ be    such that
\begin{align}\label{3.8c}  
\rho_\alpha =h\left(r_\alpha\right)
\end{align}
near $\partial U_\alpha$,
  ${\rm Supp}(\rho_\alpha)\subseteq \overline{U}_\alpha$ and 
\begin{align}\label{3.8a}
 \rho_\alpha>0\ \ \   {\rm over}\ \ U_\alpha.
\end{align}
The existence of $\rho_\alpha$ is clear. 

For any function or a smooth form $\rho$ on $M$,
we use the same notation $\rho $ to denote its lift $\pi^*\rho $.

For any $p\in {\bf B}$, set $U_{ p}=\cap_{p\in U_\alpha}U_\alpha$. Then there exists a (sufficiently small) open neighborhood $W_p$ of $p\in M$ with $\overline W_p\subset U_p$ such that for any different $p,\,q\in {\bf B}$, one has $\overline{W}_p\cap\overline{W}_q=\emptyset$.

\begin{rem}\label{t2.1} For any $p\in {\bf B}$, 
let $U_{\alpha_1},\,\cdots,\,U_{\alpha_{N_p}}$, with $\alpha_1>\cdots>\alpha_{N_p}$,  be the open coordinate charts  containing $p$. Without loss of generality, we assume that $\rho_{\alpha_1},\,\cdots,\,\rho_{\alpha_{N_p}}$ equal to $1$ on $W_p$. 
Also, we denote by  $U_{\beta_1},\,\cdots,\,U_{\beta_{2n}}$, with $\beta_1>\cdots>\beta_{2n}$,   the open coordinate charts such that 
$\partial U_{\beta_1},\,\cdots,\,\partial U_{\beta_{2n}}$ intersect at    $p$.  
Moreover, by making $W_p$ small enough, we may well assume that $\overline W_p\cap\overline U_\alpha=\emptyset$ for  any   $\alpha\notin \{\alpha_1,\,\cdots,\,\alpha_{N_p},\,\beta_1,\,\cdots,\,\beta_{2n}\}$.  
\end{rem}

As a final notation, we set
\begin{align}\label{2.8b}
{\bf B}_+=\left\{p\in {\bf B}\,:\, \beta_{2n}>\alpha_1  \right\}.
\end{align}

The following proposition   will be proved in Section \ref{s2.4}. 

\begin{prop}\label{t0.1} For any $p\in{\bf B_+}$, there is a sufficiently small open neighborhood $V_p\subset W_p$ of $p$ such that for any $\phi\in C^\infty(M)$ supported in $V_p$ with $\phi=1$ near $p$, the following limit   exists,
\begin{align}\label{0.1}
\nu_p=\lim_{T\rightarrow +\infty}\int_{\mF}\phi \prod_{i=1}^{2n}\left(\frac{d\rho_{\beta_i}T^{\beta_i}}{4\pi}\sum_{j=1}^{2n}d\left(y_{\beta_i}^j\right)^2\right)
\exp\left(-\sum_{i,\,j=1}^{2n}\rho_{\beta_i}
T^{\beta_i}\left(y_{\beta_i}^j\right)^2-T^{\alpha_1}\sum_{j=1}^{2n}\left(y_{\alpha_1}^j\right)^2\right).
\end{align}
\end{prop}

\begin{rem}\label{t0.4}
The local index  $\nu_p$ in (\ref{0.1}) depends essentially only on the correlations between  $\sum_{j=1}^{2n}(y_{\alpha_1}^j)^2$ and  $\sum_{j=1}^{2n}(y_{\beta_i}^j)^2$, $i=1,\,\cdots,\,2n$. For example, if there exist constants $a_i$ ($i=1,\,\cdots,\,2n$) such that
\begin{align}\label{0.4}
 \sum_{i,\,j=1}^{2n}a_i\left(y_{\beta_i}^j\right)^2=\sum_{j=1}^{2n}\left(y_{\alpha_1}^j\right)^2 
\end{align}
on $\pi^{-1}(p)$, then $\nu_p=0$. This fact might be of help when studying affine manifolds. 
\end{rem}

We can now state our main result as follows.

\begin{thm}\label{t0.2}
The following identity holds,
\begin{align}\label{0.2}
\left\langle e(F),[M]\right\rangle =\sum_{p\in{\bf B}_+}\nu_p. 
\end{align}
\end{thm}

Theorem \ref{t0.2} will be proved in Section \ref{s2.5}. 

 \begin{rem}\label{t0.4}
The Milnor example mentioned earlier shows that the index $\nu_p$ can be nonzero. On the other hand, the sum in  the right hand side of (\ref{0.2})   looks mystic. While it should be related to the \v{C}ech cohomology (at least in the case of $F=TM$),  it depends on the ordering of the coordinate charts $U_\alpha$'s. If one changes the ordering, then the set ${\bf B}_+$ changes. This  sounds  interesting  and deserves further study. 
\end{rem}

\subsection{Superconnections and flat vector bundles}\label{s2.2}

In what follows, for clarity, we will decorate elements in  $\pi^*\Lambda^*(F^*)$ 
with a \  $\widehat \cdot$ \ notation.

 Define on $\pi^{-1}(U_\alpha)$ that
\begin{align}\label{3.8}
\widehat Y=\sum_{k=1}^{2n}y_{\alpha}^{k}\widehat{\frac{\partial}{\partial y_{\alpha}^k}} \in \Gamma\left( \pi^*F  \right).
\end{align}
Then   $\widehat Y$ is a well-defined canonical section of $\pi^*F$ over the total space $\mF$ of $F$.

For any $T>0$,
 let $\widehat\eta_T\in \Gamma( \pi^* F^*)$ be defined by  
\begin{align}\label{3.9}
\widehat\eta_T=\sum_{\alpha=1}^N \rho_\alpha T^\alpha
 \sum_k y^k_\alpha\, \widehat{dy^k_\alpha}  .
\end{align}

\begin{rem}\label{t3.1} 
For any $T>0$, if we give $F$ the Euclidean metric defined by 
\begin{align}\label{3.10}
g^{F}_T=\sum_{\alpha=1}^N \rho_\alpha T^\alpha
  \sum_k \left( {dy_\alpha^k}\right)^2  
 ,
\end{align}
then   $\widehat \eta_T$ is the metric dual of $\widehat Y$ (with respect to $\pi^*g^{F}_T$). 
\end{rem}

Set as in (\ref{1.1}) that 
\begin{align}\label{3.11}
{  c}_T\left(\widehat Y\right)=\widehat\eta_T\wedge-i_{\widehat Y}.
\end{align}
Then ${  c}_T(\widehat Y)$ acts on $( \pi^*\Lambda^*(F^*))|_Y$. Moreover, one has
\begin{align}\label{3.12}
|Y|^2_{g^{F}_T}=-{ c}_T\left(\widehat Y\right)^2
=\sum_{\alpha,\,k}\rho_\alpha T^\alpha  
 \left(y^k_\alpha\right)^2  
.
\end{align}

For any $T>0$, let $A_T$ be the superconnection on $\pi^* \Lambda^*(F^*)$ defined by
\begin{align}\label{3.13}
A_T=\pi^*\nabla^{\Lambda^*(F^*)}+c_T\left(\widehat Y\right).
\end{align}

By Corollary \ref{t1.2}, one has
\begin{align}\label{3.14}
\langle e(F),[M]\rangle=\left(\frac{1}{2\pi}\right)^{2n} \int_{\mF}{\rm tr}_s\left[\exp\left(A_T^2\right)\right]. 
\end{align}

We need to compute $\int_{\mF}{\rm tr}_s[\exp(A_T^2)]$, which does not depend on $T>0$.

From    (\ref{3.8}), (\ref{3.9})  and (\ref{3.11})-(\ref{3.13}), one has
\begin{multline}\label{3.15}
A^2_T=\left[ \pi^* \nabla^{\Lambda^*(F^*)},\sum_{\alpha ,\,k}  \rho_\alpha T^\alpha
   y^k_\alpha \widehat{dy^k_\alpha}    -i_{\widehat Y}\right]
-
|Y|^2_{g^{F}_T}
\\
= \sum_{\alpha,\,k}T^\alpha  \left(d\rho_\alpha  y_\alpha^k   \widehat{dy^k_\alpha}+\rho_\alpha   dy_\alpha^k\, \widehat{dy^k_\alpha} \right)  
- \sum_{k }     d y_\alpha^k\otimes i_{\widehat{\frac{\partial}{\partial y_\alpha^k}} }
-|Y|^2_{g^{F}_T}
 .
\end{multline}

Set on each $\pi^{-1}(U_\alpha)$ that
\begin{align}\label{3.16}
\widehat d^V=\sum_k \widehat{dy^k_\alpha} \frac{\partial}{\partial y_\alpha^k}.
\end{align}
Clearly,  $\widehat d^V$ is well-defined over $\mF$.

From (\ref{3.12}), (\ref{3.15}) and (\ref{3.16}), one has
\begin{align}\label{3.18}
A_T^2=  \frac{1}{2}d^H\widehat d^V\left(|Y|_{g^{F}_T}^2\right)  +\sum_{\alpha,\,k }T^\alpha\rho_\alpha  dy_\alpha^k\, \widehat{dy^k_\alpha} 
- \sum_{k }    dy_\alpha^k\otimes i_{\widehat{\frac{\partial}{\partial y_\alpha^k}} }-|Y|^2_{g^{F}_T}
 .
\end{align}

Set 
\begin{align}\label{3.19}
B_T^2=  \frac{1}{2}d^H\widehat d^V\left(|Y|_{g^{F}_T}^2\right)   - \sum_{k }    dy_\alpha^k\otimes i_{\widehat{\frac{\partial}{\partial y_\alpha^k}} }-|Y|^2_{g^{F}_T}
 .
\end{align}

By  (\ref{3.18}), (\ref{3.19}), \cite[Proposition 4.9]{BZ92}  and  a simple degree counting along vertical directions, one sees   that
\begin{align}\label{3.20}
{\rm tr}_s\left[\exp\left(A_T^2\right)\right] = 
{\rm tr}_s\left[\exp\left(B_T^2\right)\right] 
 .
\end{align}

By (\ref{3.16}), (\ref{3.19}) and \cite[Proposition 4.9]{BZ92}, we see that if we exchange $\widehat{dy_\alpha^k}$ and $dy^k_\alpha$, we get the same supertrace of $\exp(B_T^2)$. Thus, we have
\begin{multline}\label{3.24}
{\rm tr}_s\left[\exp\left(B_T^2\right)\right]= 
\left\{
\exp\left(\frac{1}{2}d^Hd^V|Y|^2_{g^{F}_T}-|Y|^2_{g^{F}_T}\right)\right\}^{(4n)}
   {\rm tr}_s\left[\exp\left(-  \sum_k\widehat {dy_\alpha^k} \, i_{\widehat{\frac{\partial}{\partial y_\alpha^k}}}\right)\right]
\\
= 
\left\{\exp\left(\frac{1}{2}d^Hd^V|Y|^2_{g^{F}_T}-|Y|^2_{g^{F}_T}\right)\right\}^{(4n)}
  .
\end{multline}

From (\ref{3.14}), (\ref{3.20}),  (\ref{3.24}) and a simple transgression argument, one gets

\begin{prop}\label{t3.2a}   
For any Euclidean metric $h^F$ on $F$, one has
\begin{align}\label{000}
\langle e(F),[M]\rangle = \left(\frac{1}{2\pi}\right)^{2n}\int_{\mF}\exp\left(\frac{1}{2}d^Hd^V|Y|^2_{h^{F}}-|Y|^2_{h^{F}}\right) . 
\end{align}
\end{prop}

$\ $

\subsection{The analysis outside of   $  {\bf B}_+$}\label{s2.3}

We continue to  work with $g_T^F$.

\begin{prop}\label{t3.2}   
 For any $p\in M\setminus   {\bf B_+} $, there is an open neighborhood $V_p$ of $p$ in $M$ such that  for any  smooth function $f\in C^\infty(M)$   supported in $V_p$, one has
\begin{align}\label{3.25}
\lim_{T\rightarrow +\infty}\int_{\mF}f\,  \exp\left(\frac{1}{2}d^Hd^V|Y|^2_{g^{F}_T}-|Y|^2_{g^{F}_T}\right)=0
 .
\end{align}
\end{prop}

\begin{proof}  
By (\ref{2.3}) and  (\ref{3.12}), one has  
\begin{align}\label{3.26}
 \frac{1}{2}d^Hd^V|Y|^2_{g^{F}_T}=\sum_{\alpha,\,k}  d\rho_\alpha  T^\alpha y_\alpha^kdy_\alpha^k
 .
\end{align}
Thus, one has
\begin{align}\label{3.27}
  \exp\left(\frac{1}{2}d^Hd^V|Y|^2_{g^{F}_T}-|Y|^2_{g^{F}_T}\right)
=
 \prod_\alpha\left(1+d\rho_\alpha  T^\alpha \sum_{k}y_\alpha^kdy_\alpha^k \right)
\exp\left( -|Y|^2_{g^{F}_T}\right)
 .
\end{align}

For simplicity, we denote
\begin{align}\label{3.27b}
h_\alpha= \sum_{k}\left(y_\alpha^k\right)^2 
.
\end{align}

Take $p\in M$.  We assume that among $\{U_\alpha\}_{\alpha=1}^N$, there are exactly $N_p$ elements $U_{{\alpha_i}}$, with $\alpha_1>\cdots>\alpha_{N_p}$,   containing $p$.

If $p$ does not lie on any boundary of $U_\alpha$'s, then from (\ref{3.27}) and (\ref{3.27b}), one has
\begin{align}\label{3.28}
 \left\{ \exp\left(\frac{1}{2}d^Hd^V|Y|^2_{g^{F}_T}-|Y|^2_{g^{F}_T}\right)\right\}^{(4n)}=\sum_{\{\alpha_{i_j}\}}\prod_{j=1}^{2n} \left(d\rho_{\alpha_{i_j}} T^{\alpha_{i_j}}\frac{dh_{\alpha_{i_j}}}{2}\right)  \exp\left( -|Y|^2_{g^{F}_T}\right)
\end{align}
near $\pi^{-1}(p)$, 
where $\alpha_{i_j}$ runs through $\alpha_i$, $1\leq i\leq N_p$.  Each $\alpha_{i_j}$  appears at most once in a product. Moreover, 
by (\ref{3.8a}) and  (\ref{3.12}),   one has near $\pi^{-1}(p)$ that
\begin{align}\label{3.30}
 |Y|^2_{g^{F}_T}  
\geq
\frac{1}{2}\,\rho_{\alpha_1}(p)\,T^{\alpha_1}\sum_{k}\left(y_{\alpha_{1}}^k\right)^2 ,
\end{align}
where $\rho_{\alpha_1}(p)>0$. From (\ref{3.30}), one sees that  there is an open neighborhood $V_p$ of $p\in M$ such that for any $f\in C^\infty(M)$ supported in $V_p$,  when $T>>0$,
\begin{align}\label{3.31}
\left|
\int_{\mF}f\,
\prod_{j=1}^{2n} \left(d\rho_{\alpha_{i_j}} T^{\alpha_{i_j}}{dh_{\alpha_{i_j}}}{}\right)  \exp\left( -|Y|^2_{g^{F}_T}\right) \right|
=O\left(  \prod_{j=1}^{2n}T^{\alpha_{i_j}-\alpha_1}\right)
=O\left(\frac{1}{T^{n(2n-1)}}\right).
\end{align}

Formula  (\ref{3.25}) follows from   (\ref{3.28}) and  (\ref{3.31}).


We now assume that  $p$ lies on the boundaries of  some $U_\alpha$'s.
To be more precise, we assume that $p$ lies on the boundaries of $\{U_{\beta_i}\}_{i=1}^{M_p}$ with $\beta_1>\beta_2>\cdots>\beta_{M_p}$.  Then by    (\ref{3.27}),  the terms we need to consider, near $\pi^{-1}(p)$, are of the form
\begin{align}\label{3.32}
\prod_{h=1}^{k} \left(d\rho_{\beta_{i_h}} T^{\beta_{i_h}}dh_{\beta_{i_h}}\right)  
\prod_{j=1}^{2n-k} \left(d\rho_{\alpha_{i_j}} T^{\alpha_{i_j}}dh_{\alpha_{i_j}}\right)  \exp\left( -|Y|_{g^{F}_T}^2  \right).
\end{align}

For simplicity, we assume that $\beta_{i_1}>\cdots>\beta_{i_k}$ and $\alpha_{i_1}>\dots>\alpha_{i_{2n-k}}$.

By (\ref{3.12}), there exists  constant $c_1>0$ such that  
 one has, near $\pi^{-1}(p)$,
\begin{align}\label{3.36}  
|Y|_{g^{F}_T}^2\geq
c_1\left( 
\sum_{j=1}^{k}\rho_{\beta_{i_j}}T^{\beta_{i_j}} +\rho_{\alpha_1}(p)\,T^{\alpha_1} \right)\sum_{l}\left(y_{\alpha_{1}}^l\right)^2 .
\end{align}

From (\ref{3.36}), one sees that there exists constant $C_1>0$ such that when $T>>0$, the following formula holds near $p\in M$,
\begin{multline}\label{3.37}
 \left|\int_{\mF/M}
\prod_{h=1}^{k} \left(d\rho_{\beta_{i_h}} T^{\beta_{i_h}}dh_{\beta_{i_h}}\right)  
\prod_{j=1}^{2n-k} \left(d\rho_{\alpha_{i_j}} T^{\alpha_{i_j}}dh_{\alpha_{i_j}}\right)  \exp\left( -|Y|_{g^{F}_T}^2  \right)\right|
\\
\leq C_1  \left|\frac{\prod_{h=1}^k\left(d\rho_{\beta_{i_h}}T^{\beta_{i_h} }\right)\prod_{j=1}^{2n-k}\left(d\rho_{\alpha_{i_j}}T^{\alpha_{i_j} }\right)}
{\left( 
\sum_{j=1}^{k}\rho_{\beta_{i_j}}T^{\beta_{i_j} } +\rho_{\alpha_1}(p)T^{\alpha_1 } \right)^{2n}} \right|,
\end{multline}
where for a form $Fd{\rm vol}$, we use the notation $|Fd{\rm vol}|=|F|d{\rm vol}$, and $|Fd{\rm vol}| \leq |Gd{\rm vol}|$ means $|F|\leq |G|$. 

Recall that the boundaries of $U_{\beta_{i_h}}$'s intersect transverally to each other.

Since the function $h$ used in  the definition of $\rho_\alpha$'s in  (\ref{3.8c})  is 
increasing, one has
\begin{multline}\label{3.40a}
  \left|\frac{\prod_{h=1}^k\left(d\rho_{\beta_{i_h}}T^{\beta_{i_h} }\right) \prod_{j=1}^{2n-k}\left(d\rho_{\alpha_{i_j}}T^{\alpha_{i_j} }\right)}
{\left( 
\sum_{j=1}^{k}\rho_{\beta_{i_j}}T^{\beta_{i_j} } +\rho_{\alpha_1}(p)T^{\alpha_1 } \right)^{2n}} \right|
\leq 
 \left|\prod_{h=1}^k\frac{ d\rho_{\beta_{i_h}}T^{\beta_{i_h} } }
 {\rho_{\beta_{i_h}}T^{\beta_{i_h} } +\rho_{\alpha_1}(p)T^{\alpha_1 }  } \cdot 
\prod_{j=1}^{2n-k}\frac{ d\rho_{\alpha_{i_j}}T^{\alpha_{i_j} } }
{
\rho_{\alpha_1}(p)T^{\alpha_1 }  } \right|
\\
= \left|\left(\prod_{h=1}^k  d \log\left(  
 \rho_{\beta_{i_h}}  +\rho_{\alpha_1}(p)T^{\alpha_1-\beta_{i_h} } \right)\right)   \cdot 
\prod_{j=1}^{2n-k}\frac{ d\rho_{\alpha_{i_j}}T^{\alpha_{i_j} } }
{
\rho_{\alpha_1}(p)T^{\alpha_1 }  } \right|
 .
\end{multline}

It is easy to see  that for $a>0$ sufficiently small and $T>>0$, the  integration (for each $1\leq h\leq k$) along $0\leq r_{\beta_{i_h}}\leq a$   of    $d \log(  
 \rho_{\beta_{i_h}}  +\rho_{\alpha_1}(p)T^{\alpha_1 -\beta_{i_h}}  )$ is of $O(\log T)$ (resp. $O(\frac{1}{T})$)   if $\beta_{i_h}>\alpha_1$ (resp. $\beta_{i_h}<\alpha_1$). Also, if $k\leq 2n-2$, then one has that, in view of (\ref{3.31}),    when $T\geq 1$,
\begin{align}\label{3.40b}  
 \prod_{j=1}^{2n-k}\frac{  T^{\alpha_{i_j} } }
{
 T^{\alpha_1 }  } \leq \frac{1}{T} . 
\end{align}

From  (\ref{3.37})-(\ref{3.40b}), one finds that if $\alpha_1>\beta_{i_k}$ or if $k\leq 2n-2$, then there exists a sufficiently small open neighborhood $V_p$ of $p\in M$ such that for any smooth function $f\in C^\infty(M)$ supported in $V_p$, 
\begin{align}\label{3.41}  
 \lim_{T\rightarrow +\infty}\int_{\mF}f\,
\prod_{h=1}^{k} \left(d\rho_{\beta_{i_h}} T^{\beta_{i_h}}dh_{\beta_{i_h}}\right)  
\prod_{j=1}^{2n-k} \left(d\rho_{\alpha_{i_j}} T^{\alpha_{i_j}}dh_{\alpha_{i_j}}\right) 
\exp\left( -|Y|_{g^{F}_T}^2  \right) 
=0.
\end{align}

We need only to consider the case of $k=M_p=2n-1$ with $\beta_{2n-1}>\alpha_1$, and the case of $M_p=2n$. 
For the case of $k=M_p=2n-1$ and $\beta_{2n-1}>\alpha_1$ in (\ref{3.32}),  if $\alpha_{i_1}<\alpha_1$, then by (\ref{3.40a}) one still gets (\ref{3.41}). Thus, we need only to deal with the term 
\begin{multline}\label{3.42}
\prod_{h=1}^{2n-1} \left(d\rho_{\beta_{h}} T^{\beta_{h}}dh_{\beta_{h}}\right)  
 \left(d\rho_{\alpha_{1}} T^{\alpha_{1}}dh_{\alpha_{1}}\right)  \exp\left( -|Y|_{g^{F}_T}^2  \right)
\\
=d^V\left(\prod_{h=1}^{2n-1} \left(d\rho_{\beta_{h}} T^{\beta_{h}}dh_{\beta_{h}}\right)  
 \rho^{-1}_{\alpha_1}d\rho_{\alpha_{1}}  \exp\left( -|Y|_{g^{F}_T}^2  \right)\right)
\\
- \prod_{h=1}^{2n-1} \left(d\rho_{\beta_{h}} T^{\beta_{h}}dh_{\beta_{h}}\right)  
 \rho^{-1}_{\alpha_1}d\rho_{\alpha_{1}} \left(\sum_{i=2}^{N_p}  \rho_{\alpha_{i}} T^{\alpha_{i}}dh_{\alpha_{i}}\right)\exp\left( -|Y|_{g^{F}_T}^2  \right)
,
\end{multline}
from which and the Stokes formula, we still get   (\ref{3.41}) via (\ref{3.40a}). 

For the case of $M_p=2n$,   by Remark \ref{t2.1} one has $\rho_{\alpha_j}=1$ ($j=1,\,...,\,N_p$) near $p\in M$. Thus one may   assume $k=2n$.  Since  $\beta_{2n}(p)<\alpha_1(p)$, (\ref{3.41}) follows from (\ref{3.40a}).

Thus (\ref{3.41}), which implies (\ref{3.25}),  holds for $p\notin{\bf B}_+$. 
\end{proof}

\subsection{Proof of Proposition \ref{t0.1}}\label{s2.4}

We now suppose $p\in {\bf B}_+$.  
Recall that    $p$ is an intersection point of $\{\partial U_{ \beta_i}\}_{i=1}^{2n}$ with $  \beta_1>\cdots > \beta_{2n}$, and that $p$ lies in the open coordinate charts $\{U_{ \alpha_j}\}_{j=1}^{N_p}$ with $ \alpha_1>\cdots> \alpha_{N_p}$. Moreover, $ \beta_{2n}> \alpha_1$. 

For brevity, we set the notation on $\pi^{-1}(W_p)$ that
\begin{align}\label{3.47}
  |Y|_T^2= \sum_{i=1}^{2n}\rho_{\beta_{i}}T^{\beta_i}h_{\beta_i} + T^{\alpha_1}h_{\alpha_1}.  
\end{align}

Let $0\leq \phi\leq 1$ be a smooth function on $M$ such that ${\rm Supp}(\phi)\subset V_p$ where $V_p\subset   W_p$ is a sufficiently small open neighborhood of $p$, and that $\phi$ equals to $1$ near  $p\in M$.   
We need only to prove that the  following limit exists,  
\begin{align}\label{3.43}
\lim_{T\rightarrow+\infty}\int_{\mF}\phi\prod_{j=1}^{2n} \left(d\rho_{\beta_{j}} T^{\beta_{j}}dh_{\beta_{j}}\right)  
  \exp\left( -|Y|_{T}^2  \right)
.
\end{align}

For any $T>0$, set 
\begin{align}\label{3.50}
\gamma_T=\frac{\partial }{\partial T} \int_{\mF}\phi \prod_{j=1}^{2n} \left(d\rho_{\beta_{j}} T^{\beta_{j}}dh_{\beta_{j}}\right)  
  \exp\left( - |Y|_T^2  \right).
\end{align}

\begin{lemma}\label{t3.5}
When $T>0$ is very large, one has
\begin{align}\label{3.51}
\gamma_T=O\left(\frac{(\log T)^{2n-1}}{T^2}\right).
\end{align}
\end{lemma}
\begin{proof}
Set  on any $\pi^{-1}(U_\alpha)$ that $i_Y=\sum_{k}y_\alpha^k\,i_{\frac{\partial}{\partial y_\alpha^k}}$, which does not depend on $\alpha$. 
 Then one has
\begin{align}\label{3.45}
  \left(d^H- i_Y\right)^2=0
.
\end{align}

For any $1\leq \alpha\leq N$, one verifies   that
\begin{align}\label{3.45a}
  \left(d^H- i_Y\right)\left(\rho_\alpha dh_\alpha\right)=d\rho_\alpha dh_\alpha-2\,\rho_\alpha h_\alpha.  
\end{align}

By (\ref{3.47}),   (\ref{3.50}),  (\ref{3.45}),  (\ref{3.45a}) and the Stokes formula, one has
\begin{multline}\label{3.52}
 \gamma_T =2^{2n} \frac{\partial }{\partial T}
\int_{\mF}  \phi \prod_{i=1}^{2n} \left(e^{\left(d^H-i_Y\right) \left(\rho_{\beta_{i}}T^{\beta_i}dh_{\beta_i}/2\right) }\right)
e^{\left(d^H-i_Y\right) T^{\alpha_1}dh_{\alpha_1} /2  }   
\\
=   \frac{2^{2n-1}}{T}
\int_{\mF}  \phi \left(d^H-i_Y\right)\left(\left(\sum_{i=1}^{2n}\beta_i\rho_{\beta_{i}}T^{\beta_i}dh_{\beta_i} +\alpha_1T^{\alpha_1}dh_{\alpha_1}\right)\right.
\\
\left.
\cdot \left(
\prod_{i=1}^{2n} e^{\left(d^H-i_Y\right) \left(\rho_{\beta_{i}}T^{\beta_i}dh_{\beta_i}/2\right) }\right)
e^{\left(d^H-i_Y\right) T^{\alpha_1}dh_{\alpha_1} /2  }  \right)
\\
= - \frac{1}{T}
\int_{\mF} d \phi  \left(\sum_{i=1}^{2n}\beta_i\rho_{\beta_{i}}T^{\beta_i}dh_{\beta_i}
\prod_{j\neq i}   \left(d\rho_{\beta_{j}}T^{\beta_j}dh_{\beta_j} \right)\right)  \exp\left( - |Y|_T^2  \right)
\\
-   \frac{\alpha_1}{T}
\int_{\mF} d \phi\, T^{\alpha_1}dh_{\alpha_1}\left(\sum_{i=1}^{2n}
\prod_{j\neq i}   \left(d\rho_{\beta_{j}}T^{\beta_j}dh_{\beta_j}\right) \right) 
 \exp\left( - |Y|_T^2  \right)
\\
= \frac{1}{T}
\int_{\mF} d \phi \, T^{\alpha_1}dh_{\alpha_1} \left(\sum_{i=1}^{2n}\beta_i 
\prod_{j\neq i}   \left(d\rho_{\beta_{j}}T^{\beta_j}dh_{\beta_j} \right)\right)  \exp\left( - |Y|_T^2  \right)
\\
-   \frac{\alpha_1}{T}
\int_{\mF} d \phi\, T^{\alpha_1}dh_{\alpha_1}\left(\sum_{i=1}^{2n}
\prod_{j\neq i}   \left(d\rho_{\beta_{j}}T^{\beta_j}dh_{\beta_j}\right) \right) 
 \exp\left( - |Y|_T^2  \right)
,
\end{multline}
where the last equality follows from a vertical transgression argument (cf. (\ref{3.42})).

For any $q\in {\rm Supp}(d\phi)$, either one of $\rho_{\beta_i}(q)> 0$, or one of $\rho_{\beta_i}$'s vanishes near $q$. 
In the former case, since $\beta_i>\alpha_1$, by proceeding as in (\ref{3.40a}), one sees that there is a small open neighborhood $V_q\subset W_p$ of $q$  such that for any $f\in C^\infty(M)$ with ${\rm Supp}(f)\subset V_q$, when $T>1$ is large enough, one has
\begin{multline}\label{3.53}
\int_{\mF}f\, d \phi \, T^{\alpha_1}dh_{\alpha_1} \left(\sum_{i=1}^{2n}\beta_i 
\prod_{j\neq i}   \left(d\rho_{\beta_{j}}T^{\beta_j}dh_{\beta_j} \right)\right)  \exp\left( - |Y|_T^2  \right)
\\
-   {\alpha_1}{}
\int_{\mF}f\, d \phi\, T^{\alpha_1}dh_{\alpha_1}\left(\sum_{i=1}^{2n}
\prod_{j\neq i}   \left(d\rho_{\beta_{j}}T^{\beta_j}dh_{\beta_j}\right) \right) 
 \exp\left( - |Y|_T^2  \right) 
=O\left(\frac{(\log T)^{2n-1}}{T}\right)
,
\end{multline}
while in the later case, by an easy vertical transgression argument, one has
\begin{multline}\label{3.54}
\int_{\mF/M} d \phi \, T^{\alpha_1}dh_{\alpha_1} \left(\sum_{i=1}^{2n}\beta_i 
\prod_{j\neq i}   \left(d\rho_{\beta_{j}}T^{\beta_j}dh_{\beta_j} \right)\right)  \exp\left( - |Y|_T^2  \right)
\\
-   {\alpha_1}{}
\int_{\mF/M} d \phi\, T^{\alpha_1}dh_{\alpha_1}\left(\sum_{i=1}^{2n}
\prod_{j\neq i}   \left(d\rho_{\beta_{j}}T^{\beta_j}dh_{\beta_j}\right) \right) 
 \exp\left( - |Y|_T^2  \right) 
=0
\end{multline}
near $q\in M$.

From (\ref{3.52})-(\ref{3.54}) and a simple partition of unity argument, one gets (\ref{3.51}). 
\end{proof}

From (\ref{3.50}) and Lemma \ref{t3.5}, one sees that
\begin{multline}\label{3.55}
 \lim_{T\rightarrow +\infty}\int_{\mF}\phi \prod_{j=1}^{2n} \left(d\rho_{\beta_{j}} T^{\beta_{j}}dh_{\beta_{j}}\right)  
  \exp\left( - |Y|_T^2 \right) = \int_{\mF}\phi \prod_{j=1}^{2n} \left(d\rho_{\beta_{j}}  dh_{\beta_{j}}\right)  
  \exp\left( - |Y|_{T=1}^2 \right)
\\
+ \lim_{T\rightarrow+\infty}\int_1^T\gamma_t\,dt
\end{multline}
exists, which completes the proof of Proposition \ref{t0.1}.

\subsection{Proof of Theorem  \ref{t0.2}}\label{s2.5}
First still assume  that $p\in {\bf B}_+$.  
Recall that by Remark \ref{t2.1}, one has that $\rho_{\alpha_j}=1$ ($1\leq j\leq N_p$) on   $W_p$. 
 Then one has near $\pi^{-1}(p) $ that 
\begin{align}\label{3.44}
 |Y|_{g^{F}_T}^2  =\sum_{i=1}^{2n}\rho_{\beta_{i}}T^{\beta_i}h_{\beta_i}+\sum_{i=1}^{N_p}T^{\alpha_i}h_{\alpha_i}.
\end{align}

We need only to deal with the term 
\begin{align}\label{3.44z}
 \lim_{T\rightarrow +\infty}\int_{\mF}\phi \prod_{i=1}^{2n} \left(d\rho_{\beta_{i}} T^{\beta_{i}}dh_{\beta_{i}}\right)  
  \exp\left( -|Y|_{g^{F}_T}^2  \right) ,
\end{align}
which is examined  in the following lemma.

\begin{lemma}\label{t3.3}
The following identity holds,
\begin{multline}\label{3.44a}
 \lim_{T\rightarrow +\infty}\int_{\mF}\phi \prod_{i=1}^{2n} \left(d\rho_{\beta_{i}} T^{\beta_{i}}dh_{\beta_{i}}\right)  
  \exp\left( -|Y|_{g^{F}_T}^2  \right) 
\\
=  \lim_{T\rightarrow +\infty}\int_{\mF}\phi \prod_{i=1}^{2n} \left(d\rho_{\beta_{i}} T^{\beta_{i}}dh_{\beta_{i}}\right)  
  \exp\left( - |Y|_T^2 \right)
.
\end{multline}
\end{lemma}

\begin{proof} 

 From  (\ref{3.47}), (\ref{3.45}), (\ref{3.45a}) and  (\ref{3.44}),  one deduces that
\begin{multline}\label{3.46}
 \frac{ \phi}{2^{2n}} \prod_{i=1}^{2n} \left(d\rho_{\beta_{i}} T^{\beta_{i}}dh_{\beta_{i}}\right)  
  \exp\left( -|Y|_{g^{F}_T}^2  \right) 
-  \frac{ \phi}{2^{2n}} \prod_{i=1}^{2n} \left(d\rho_{\beta_{i}} T^{\beta_{i}}dh_{\beta_{i}}\right)  
  \exp\left( -|Y|_T^2 \right)
\\
=\left\{ \phi \prod_{i=1}^{2n} \left(e^{\left(d^H-i_Y\right) \left(\rho_{\beta_{i}}T^{\beta_i}dh_{\beta_i}/2\right) }\right) e^{- T^{\alpha_1}h_{\alpha_1}} 
   \left( e^{\left(d^H-i_Y\right) \sum_{i=2}^{N_p}T^{\alpha_i}dh_{\alpha_i} /2} - 1  \right) \right\}^{(4n)}
\\
=\left\{\left( d^H-i_Y\right) \left(\phi\,  \prod_{i=1}^{2n} \left(e^{\left(d^H-i_Y\right) \left(\rho_{\beta_{i}}T^{\beta_i}dh_{\beta_i}/2\right)}\right) e^{- T^{\alpha_1}h_{\alpha_1}} 
   \left( \sum_{i=2}^{N_p}T^{\alpha_i}dh_{\alpha_i} /2\right) 
\frac{e^{-\sum_{i=2}^{N_p}T^{\alpha_i}h_{\alpha_i}}  - 1}{-\sum_{i=2}^{N_p}T^{\alpha_i}h_{\alpha_i}  } \right)\right.
\\
- \left. d\phi\,  \prod_{i=1}^{2n} \left(e^{\left(d^H-i_Y\right) \left(\rho_{\beta_{i}}T^{\beta_i}dh_{\beta_i}/2\right)}\right) e^{- T^{\alpha_1}h_{\alpha_1}} 
   \left( \sum_{i=2}^{N_p}T^{\alpha_i}dh_{\alpha_i}/2 \right) 
\frac{e^{-\sum_{i=2}^{N_p}T^{\alpha_i}h_{\alpha_i}}  - 1}{-\sum_{i=2}^{N_p}T^{\alpha_i}h_{\alpha_i}  }   
\right\}^{(4n)}
.
\end{multline}

From (\ref{3.47}), (\ref{3.45a}), (\ref{3.46}) and the Stokes formula, one gets
\begin{multline}\label{3.48}
 \int_{\mF}\phi \prod_{h=1}^{2n} \left(d\rho_{\beta_{h}} T^{\beta_{h}}dh_{\beta_{h}}\right)  
  \exp\left( -|Y|_{g^{F}_T}^2  \right) 
-   \int_{\mF}\phi \prod_{h=1}^{2n} \left(d\rho_{\beta_{h}} T^{\beta_{h}}dh_{\beta_{h}}\right)  
  \exp\left( - |Y|_T^2  \right)
\\
=\int_{\mF}d\phi\left(\sum_{h=1}^{2n}  \prod_{i\neq h} \left( d  \rho_{\beta_{i}}T^{\beta_i}dh_{\beta_i} \right)\right) 
   \left( \sum_{i=2}^{N_p}T^{\alpha_i}dh_{\alpha_i} \right) 
\frac{e^{-\sum_{i=2}^{N_p}T^{\alpha_i}h_{\alpha_i}}  - 1}{\sum_{i=2}^{N_p}T^{\alpha_i}h_{\alpha_i}  }    \exp\left( - |Y|_T^2  \right)
.
\end{multline}

Recall that  $\alpha_1>\alpha_i$ for $i\geq 2$. Also,  for any $t\geq 0$, one has
\begin{align}\label{3.49}
0\leq\frac{1-e^{-t}}{t}\leq 1.
\end{align}

By (\ref{3.47}), (\ref{3.48}), (\ref{3.49})  and proceeding as in (\ref{3.40a}), one gets (\ref{3.44a}). 
\end{proof}

Now   to prove Theorem \ref{t0.2}, 
we choose a finite selection of $V_p$'s in Propositions \ref{t0.1} and  \ref{t3.2} so that they form an open covering of $M$. Let $\{f_{V_p}\}$'s be a partition of unity subordinate to this open covering.  We  assume that each    $p\in  {\bf B_+} $ is covered by only one $V_p$ on which $f_{V_p}=1$ near $p$.  Since ${\bf B}_+$ consists of finite points, the existence of such a covering and partition of unity is clear.  

   Theorem \ref{t0.2} then follows from Propositions \ref{t0.1}, \ref{t3.2a}, \ref{t3.2} and Lemma \ref{t3.3} as follows, 
\begin{align}\label{3.66}
 \langle e(F),[M]\rangle =\left(\frac{1}{2\pi}\right)^{2n}\sum\lim_{T\rightarrow +\infty}\int_{\mF}f_{V_p}\exp\left(\frac{1}{2}d^Hd^V|Y|^2_{g^{F}_T}-|Y|^2_{g^{F}_T}\right) =\sum_{q\in{\bf B}_+}\nu_q .
\end{align}

 $\ $

\noindent{\bf Acknowledgments.} We thank   Fei Han, Kefeng Liu, Xiaonan Ma, Shu Shen and Guangxiang Su 						for helpful discussions.

\end{document}